\documentclass[12pt,a4paper,leqno]{amsart}
\usepackage[utf8]{inputenc}
\usepackage{amssymb,color,tikz, cancel}
\usepackage{amsmath,amscd,amsfonts,amsthm,verbatim}
\usepackage{tikz-cd}
\usepackage{enumitem}
\usepackage[all]{xy}

\usepackage{hyperref}

\usepackage{array}
\usepackage{graphicx,color}

\usepackage{geometry}
\newtheorem{theorem}{Theorem}[section]
\newtheorem{proposition}[theorem]{Proposition}

\newtheorem{lemma}[theorem]{Lemma}
\newtheorem{corollary}[theorem]{Corollary}
\newtheorem{definition}[theorem]{Definition}

\newtheorem{remark}[theorem]{Remark}
\newtheorem{conjecture}[theorem]{Conjecture}

\newtheorem{question}[theorem]{Question}

\newcommand{\ZZ}{\mathbb{Z}}

\newcommand {\PP}{\mathbb{P}}

\newcommand{\kk}{\mathbb{K}}

\newcommand{\cE}{\mathcal{E}}
\newcommand{\cF}{\mathcal{F}}
\newcommand{\cG}{\mathcal{G}}

\newcommand{\cM}{\mathcal{M}}

\newcommand{\cO}{\mathcal{O}}

\newcommand{\cU}{\mathcal{U}}

\DeclareMathOperator{\HH}{H}

\DeclareMathOperator{\End}{End}
\DeclareMathOperator{\Ext}{Ext}

\DeclareMathOperator{\rk}{rk}
\DeclareMathOperator{\rank}{rank}
\DeclareMathOperator{\Spl}{Spl}

\begin{document}

\title[Stability of syzygy bundles of Ulrich bundles]{Stability of syzygy bundles of Ulrich bundles}
  \author[R.\ M.\ Mir\'o-Roig]{Rosa M.\ Mir\'o-Roig}
  \address{Facultat de
  Matem\`atiques i Inform\`atica, Universitat de Barcelona, Gran Via des les
  Corts Catalanes 585, 08007 Barcelona, Spain} \email{miro@ub.edu, ORCID 0000-0003-1375-6547}

\subjclass[2020]{14J60, 14D20, 14F06}

\keywords{stability, syzygy bundles, Ulrich bundles}

\thanks{The author has been partially supported by the grant PID2024-157142NB-I00.}

\begin{abstract}
Let $X$ be  either a smooth K3 surface or  a smooth Fano variety (i.e. $-K_X$ is ample) of dimension $n$ and   index $i_X\ge n-2$ and let $\cE$  be an initialized Ulrich bundle on $X$. In this paper, we show that  the syzygy bundle $S_{\cE}$, defined as the kernel of the evaluation map $eval:H^{0}(X,\cE)\otimes\cO_{X}\rightarrow \cE,$ is semistable.
\end{abstract}

\maketitle

\section{Introduction}
The goal of this short note is to study the stability of syzygy bundles associated to Ulrich bundles on smooth irreducible K3 surfaces and smooth irreducible Fano varieties.

Let $X$  be a smooth irreducible projective variety of  dimension $n\ge 1$ defined over an algebraically closed field $\kk$  of characteristic 0. Let $\cE$
 be a globally generated vector bundle on $X$. The {\em syzygy bundle}  of $\cE$  is defined as  the kernel $S_{\cE}$ of the  evaluation map $$eval: H^0(X,\cE)\otimes{\mathcal O}_X\to \cE,$$ which is surjective. Arising in a variety of geometric and algebraic problems, syzygy bundles  have been intensively studied from different points of view ranging from the syzygies of $X$ (and the $N_p$ properties in the sense of Green \cite{G}) to questions of tight closure.
 In particular, when $\cE=\cO _X(L)$ is a globally generated line bundle,  many efforts have been invested in knowing whether  the associated syzygy bundle $S_L:=S_{\cO_X(L)}$ which now fits inside the exact sequence:
 $$
 0\longrightarrow S_L\longrightarrow H^0(X,\cO _X(L))\otimes \cO_{X}\longrightarrow \cO_{X}(L)\longrightarrow 0,
 $$
 is stable with respect to a fixed  ample line bundle $H$ on $X$  (see, for instance, \cite{BP}, \cite{CL} and \cite{MR}). When $X$ is a smooth curve of genus $g\ge 1$, the situation is well understood thanks to work of Butler \cite{B94} and Ein-Lazarsfeld \cite{EL}. In particular, $S_L$ is stable provided $deg(L)\ge 2g+1$. Motivated by this problem, in
 \cite[Corollary 2.6 and Conjecture 4.14]{ELM} Ein-Lazarsfeld-Mustopa established the following Conjecture:

\begin{conjecture} \label{ConjELM}
Let $A$ and $P$ be two line bundles on a smooth projective variety $X$. Assume that $A$ is very ample and set $L_{d}:=dA+P$ for any positive integer $d$. Then, the syzygy bundle $S_{L_{d}}$ is $A-$stable for $d\gg0$.
\end{conjecture}

 Notice that the semistability of $S_L$ with respect to $L$ is equivalent to the semistability of the pull back $\varphi _L^{*}T_{\PP}$ of the tangent bundle of $\PP := \PP(\HH ^0(X,L)^{\vee })$, where $\varphi _{|L|}:X\longrightarrow \PP(\HH ^0(X,L)^{\vee })
$ is the morphism defined by the linear system  $|L|$. Conjecture \ref{ConjELM} has been  solved in \cite[Theorem 4.3 and Corollary 4.4]{R}.

For vector bundles $\cE$ of higher rank very little is known. In fact,  the stability of the syzygy bundle $S_{\cE}$ has been recently  studied for globally generated vector bundles $\cE$ of higher rank on smooth projective varieties of arbitrary dimension $n\ge 1$ and so far only general results for curves and surfaces are known. Indeed, in \cite{B94} Butler proves that the syzygy bundle $S_{\cE} $ associated to a semistable globally generated vector bundle $\cE$ on a smooth curve $C$ of genus $g\ge 1$ is semistable provided $\mu(\cE)\ge 2g$ and even stable unless $\mu(\cE)=2g$ and $C$ is either hyperelliptic or $\Omega _C\hookrightarrow \cE$. For the stability of a syzygy bundle $S_{\cE}$ associated to a stable vector bundle $\cE$ on a surface  the reader can look at \cite{BP} and \cite{MRay} where the following questions are investigated:

\begin{question}\label{keyquestion} \rm Let $X$ be a smooth projective variety of dimension $n\ge 1$, $H$ an ample line bundle on $X$ and $\cE$ a stable globally generated vector bundle on $X$.

(i)  Is  $S_{\cE }$ stable with respect to $H$?

(ii) Is there an integer $m\gg 0$  such that  $S_{\cE (m)}$ is stable with respect to $H$?
\end{question}

The goal of this paper is to address  the above question for Ulrich bundles $\cE$ on K3 surfaces and on Fano varieties $X$ of dimension $n\ge 1$ and index $i_X\ge n-2$.
 In section \ref{Section:Preliminaries}, we collect all the preliminary results needed to establish  our main result. We have subdivided this section into two subsections. First, in Subsection \ref{Subsection:stability}, we recall the basic notions of semistability that will be used throughout the paper. In Subsection \ref{Subsection:Ulrich},  in order to provide a more self-contained exposition, we review the essential definitions and results concerning Ulrich bundles. The core of this note is presented in Section \ref{Core}
where we deal with   Question \ref{keyquestion}  for
the syzygy bundles associated to Ulrich bundles on smooth irreducible K3 surfaces, as well as on any $n$-dimensional smooth irreducible
 Fano variety $X$ of index
 $i_{X} \ge n-2$.
Finally, in the last part of this section, we  prove (Theorem \ref{Theorem:moduli}) that, in this setting, semistable syzygy bundles associated to unobstructed Ulrich bundles are smooth points in their moduli space. Moreover,  we explicitly compute the dimension of the irreducible component containing these points.

\section{Basic results}\label{Section:Preliminaries}

Let $X$ be a smooth irreducible projective variety of dimension $n$, defined over an algebraically closed field $\kk$ of characteristic zero and let
 $\cE$ be a  globally generated vector bundle of rank $r$ on $X$. We define the {\em syzygy bundle} $S_{\cE}$ associated to $\cE$ as the kernel of the surjective evaluation map
 \[
 ev:H^0(X,\cE)\otimes {\cO}_X \longrightarrow \cE.
 \]
Therefore, $S_{\cE}$ is a vector bundle on $X$ fitting in the following short exact sequence
\begin{equation}\label{Eq:Syzygy bundle exact sequence}
0 \longrightarrow S_{\cE} \longrightarrow H^0(X,\cE)\otimes {\cO}_X \longrightarrow \cE\longrightarrow 0. \end{equation}
In particular, we have:

\vskip 2mm
\begin{itemize}
    \item $c_1(S_{\cE})=-c_1(\cE)$,
    \item $\rk(S_{\cE})=\dim H^0(X,\cE)-r$,
    \item $c_t(S_{\cE})\cdot c_t(\cE)=1,$
    \item $\mu (S_{\cE})=\frac{c_1(S_{\cE})\cdot H^{n-1}}{\rk(S_{\cE})}=\frac{-c_1(\cE)\cdot H^{n-1}}{h^0(\cE)- \rk(\cE)}$.
\end{itemize}

\vskip 4mm
Our goal is to study the semistability of the syzygy bundle $S_{\cE}$ when $\cE$ is an Ulrich bundle of rank $r$ on either a smooth K3 surface or a smooth   Fano variety $X$ of dimension $n$ and index $i_X\ge n-2$ (see Definitions \ref{defFano} and \ref{defK3}) and to obtain local information on the geometry of their corresponding moduli spaces.

\subsection{Stability of syzygy bundles}\label{Subsection:stability}

In this paper, we are interested in determining the semistability of syzygy bundles. Let us start recalling  the  definition of stability due to Mumford-Takemoto and the key results.

\begin{definition}\rm Let $X$ be a  smooth irreducible projective variety of dimension $n$ and $H$ a very ample line bundle on $X$. A rank $r$ vector bundle $\cE$ on $X$ is said to be stable (resp. semistable) with respect to $H$ if for any subsheaf $\cF\subset \cE$ with $0<\rk(\cF)<\rk(\cE)$, we have

$$\mu (\cF)=\frac{c_1(\cF)H^{n-1}}{\rk(\cF)}<\mu (\cE)=\frac{c_1(\cE)H^{n-1}}{\rk(\cE)} $$
$$\text{( resp. } \quad \mu(\cF)= \frac{c_1(\cF)H^{n-1}}{\rk(\cF)}\le \mu (\cE)=\frac{c_1(\cE)H^{n-1}}{\rk(\cE)}\text{ )}. $$
\end{definition}

 The stability of syzygy bundles associated to globally generated line bundles $\cO_X(L)$ on polarized varieties $(X,H)$ is quite well understood (see, for instance, \cite{CL, EL, ELM, Fle, MR}) and nowadays big attention is devoted to the stability of syzygy bundles $S_{\cE}$ associated to stable globally generated vector bundles $\cE$ of higher rank $r$ on varieties of dimension $n\ge 1$. As a major challenge in this subject, Misra and Ray have recently  formulated the following question:

 \begin{question}\label{mainquestion}
     Let $X$ be a smooth irreducible complex projective variety and fix a very ample line bundle $H$ on $X$. Let $\cE$ be an $H$-semistable globally generated vector bundle of rank $r$ on $X$. Is $S_{\cE}$ stable? Is there an integer $m_0$ such that  for  $m\ge m_0$ the syzygy bundle $S_{\cE(m)}$ is $H$-stable?
 \end{question}

The question is quite open and as we already pointed out in the introduction a full answer is only known for varieties of low dimension ($n\le 2$).
  See \cite{B94} for the case of curves; and \cite[Theorem 1.3]{MRay} and  \cite[Theorem 1.1]{BP} for the case of surfaces.


\subsection{Ulrich bundles}\label{Subsection:Ulrich}

In this subsection, we introduce the vector bundles that constitute the main objects of study in this paper—namely, Ulrich bundles—and we collect the properties of these bundles that will be required in the sequel.

\begin{definition} \rm
    A coherent sheaf $\cE$ on a projective variety $X$ with a fixed ample
line bundle $\cO _X (1)$ is called {\em arithmetically Cohen-Macaulay} (for short, aCM) if it is
locally Cohen-Macaulay and $H^i (X,\cE(t)) = 0$ for all $t\in \ZZ$ and $1\le i\le  dim(X) - 1.$
\end{definition}

\begin{definition} \rm
    A coherent sheaf $\cE$ on a projective variety $X$ with a fixed ample
line bundle $\cO _X (1)$ is said to be {\em initialized}  if $H^0(X,\cE(t)) = 0$ for all $t\in \ZZ_{<0}$ but
$H^0(X,\cE)\ne 0$
\end{definition}

\begin{definition} \rm Given a smooth irreducible projective variety $X\subset \PP^n$ embedded by the very ample line bundle $\cO _X(1)$ and a vector bundle  $\cE$ on $X$, we say that $\cE$ is an {\em Ulrich bundle} if it is  an initialized aCM bundle and $h^0(X,\cE) = deg(X)rank(\cE)$ or, equivalently
$H^i (X , \cE (-jH )) = 0$ for any $i\ge 0$ and  $1 \le j \le dim X$.
\end{definition}

This definition admits several variants. For instance, an initialized vector bundle $\cE$ on a smooth irreducible projective variety $X\subset \PP^n$ embedded by the very ample line bundle $\cO _X(1)$  is an Ulrich bundle if $H^i(X,\cE (-i))=0$ for $i>0$ and $H^i(X,\cE(-i-1))=0$ for $i<dim(X)$ or, equivalently, there is a finite linear projection $\pi : X \longrightarrow \PP^n$
such that $\pi _{*} \cE = \cO _{\PP^{dim(X)}}^k$
for some positive integer $k$ \cite[Proposition 2.1]{ES}.
See \cite{ES} and \cite{CMP} for other equivalent definitions.

Ulrich bundles have made a first appearance in commutative algebra, being associated to maximal Cohen-Macaulay graded modules with maximal number of generators,
\cite{Ul}.  Eisenbud and Schreyer transferred this concept in the geometric setup and discovered a remarkable
connection with Cayley–Chow forms \cite{ES} (see also \cite{Beau} and \cite{CMP}
for recent developments from the geometric perspective and further
applications). The most significant open problem in the Ulrich bundle theory is the
Eisenbud–Schreyer conjecture which predicts that any smooth projective variety  carries an Ulrich bundle. If true, this conjecture would
have dramatic consequences on the expected shapes of cohomology
tables. While this problem is widely open, it has
been resolved in a number of relevant cases (see \cite{CMP} for more information). This note will be integrally devoted to answer Question \ref{keyquestion} for the case of  syzygy bundles $S_{\cE}$ associated to  Ulrich bundles $\cE$ on  either a smooth K3 surface or a smooth Fano  varieties of dimension $n$ and index $i_X\ge n-2$ (Theorem \ref{main}).

In the following Proposition we gather some of the properties of an Ulrich bundle $\cE$ that
will be used throughout the paper:

\begin{proposition}\label{properties}
   Let
$X\subset  \PP^N$ be a smooth variety of dimension $n$ polarized by a very ample line bundle $\cO_X(1)$ and let $\cF$ be an
Ulrich bundle on $X$ with respect to $H\in |\cO_X(1)|$. Then:
\begin{itemize}
\item[(i)] $\cF$ is globally generated;
\item[(ii)] $\cF$  is $m$-regular for any $m\ge 0$;
 \item[(iii)] $\cF$ is  semistable with respect to $H$. Moreover, if $\cF$ is strictly semistable with respect to $H$, then there exists a short exact sequence of vector bundles:
$$0 \longrightarrow \cE \longrightarrow \cF \longrightarrow \cG \longrightarrow 0,$$
with $\cE$ and $\cG$ Ulrich bundles of lower rank.
\end{itemize}
\end{proposition}
\begin{proof}
    (i) It immediately follows from the fact that any Ulrich bundle  $\cE$ on $X$  admits a linear $\cO _{\PP^N}$-resolution of the form:
\begin{equation} \label{RresolutionUlrich}
0 \longrightarrow \cO _{\PP^N}(-N+d)^{a_{N-d}}\longrightarrow \dots \longrightarrow \cO _{\PP^N}(-1)^{a_{1}}\longrightarrow \cO _{\PP^N}^{a_{0}}\longrightarrow \cE \longrightarrow 0.
\end{equation}

    (ii) The hypothesis $H^i(X,\cE(-i))=0$ for all $i> 0$ implies that $\cE$ is 0-regular and, hence, $m$-regular for all $m\ge 0$.

    (iii)  See \cite[Proposition 3.3.14 and  Corollary 3.3.17]{CMP}.
\end{proof}


\section{Stability of syzygy bundles of Ulrich bundles on Fano varieties}\label{Core}
The aim of this Section is to answer Question \ref{keyquestion}  for smooth K3 surfaces as well as for smooth Fano varieties $X$ of dimension $n$ and index $i_X\ge n-2$ and prove the semistability of the syzygy bundles $S_{\cE}$  associated to any Ulrich bundle $\cE$ on either a smooth irreducible K3 surface or a smooth irreducible Fano variety $X$ of dimension $n$ and index $i_X\ge n-2$.

\begin{definition}\label{defFano} \rm
    A {\em Fano variety} will be a smooth irreducible projective variety $X$ such that the anticanonical line bundle
$K ^{-1}_X$ is ample.
\end{definition}
Simplest examples are obtained by taking smooth complete intersections $X$ in $\PP^N$ of type
$(d_1, d_2, \cdots, d_k)$ . By adjunction formula, such a complete intersection is a Fano variety
if and only if $\sum _{i=1}^k d_i\le N$. Rational homogeneous varieties $G/H$
($G$ semisimple, $H$ parabolic) are also examples of Fano varieties. Fano varieties have a very rich internal geometry, which makes their
study very rewarding. The interest
in Fano varieties increased since Mori's program predicts that every
uniruled variety is birational to a fiberspace whose general fiber is a Fano
variety (with terminal singularities).

A basic invariant of a Fano variety is its index: The {\em index } $i_X$ of a Fano variety $X$ is the largest integer $r$ such that $-K_X=rH$ for some ample divisor $H$. It is well known that the index $i_X$ of a Fano variety $X$ of dimension $n$ is at
most $n + 1$. Moreover, if $i_X = n + 1$, then $X =\PP^n$; if $i_X = n$, then $X$
is a quadric; and  smooth Fano varieties with $i_X\ge  n-2$
can be classified (see, for example, \cite{S}). In this note we will restrict our attention to Fano varieties $X$ of dimension $n$ and  index $i_X\ge n-2$.

\begin{definition} \label{defK3} \rm
 A  {\em K3 surface}  will be a smooth irreducible projective  variety of dimension 2 with trivial canonical line bundle $\omega _X\cong \cO _X$ and  irregularity $q:=H^1(X,\cO _X)=0$.
\end{definition}

Therefore, for any K3 surface we have: arithmetic and geometric genus $p_s=p_g=1$ and Euler characteristic $\chi (\cO_X)=2$.
Basic examples of K3 surfaces are quartic surfaces in $\PP^3$,
 complete intersections of type (2,3) in $\PP^4$,
 complete intersections of type (2,2,2) in $\PP^5$, and
 double covers of $\PP^2$ branched over a sextic.

\begin{remark} \rm
   Let $X$ be a smooth irreducible Fano variety of dimension $n$ and index $i_X>1$
and let $Y\in |\cO _X(1)|$ be a general hyperplane section.
    By adjunction formula we have  $$K_Y=(K_X+Y)_{|Y}=(-i_XH+H)_{|H}=(-i_X+1)H_Y.$$ Therefore, we get that the general hyperplane section of a smooth Fano variety $X$ of dimension $n$ and index $i_X>1$ is a smooth Fano variety $Y$ of dimension $n-1$ and  index $i_Y=i_X-1$. Analogously, the general  hyperplane section of a smooth Fano 3-fold of index 1 is a smooth K3 surface.
\end{remark}

The proof of our main result relies on the following 2 lemmas.

\begin{lemma}\label{keylemma}
     Let $X$ be a smooth irreducible projective variety   with a fixed ample
line bundle $\cO _X (1)$ and let $\cE$ be an  initialized Ulrich bundle of rank $r$ on $X$.
Then, the restriction  $\cE_{|Y}$ of $\cE$ to a general hyperplane section  $Y$ of $X$ is an initialized Ulrich bundle on $Y$.
\end{lemma}

\begin{proof} Since $\cE$ is an Ulrich bundle on $X$, we have: $$H^0(X,\cE(-1))=H^1(X,\cE(-1))=H^1(X,\cE(-2))=0$$ and $h^0(X,\cE)=\deg(X)\cdot \rank(\cE).$
Therefore,  from the exact  cohomology   sequence associated to the short exact sequence
    $$
    0 \longrightarrow \cE (-1)\longrightarrow \cE  \longrightarrow \cE _{|Y } \longrightarrow 0 $$
we deduce $H^0(Y,\cE_{|Y}(-1))=0$ and $h^0(Y,\cE_{|Y})=\deg(X)\dot \rank(\cE)=\deg(Y)\cdot \rank(\cE_{|Y})$ and we conclude that $\cE_{|Y}$ is a rank $r$ Ulrich bundle on $Y$.

\end{proof}

\begin{remark} \rm  Let $X$ be a smooth irreducible projective variety  with a fixed ample
line bundle $\cO _X (1)$ and let $\cE$ be an initialized vector bundle on $X$.
   It is worthwhile to point out that, in general, the restriction of $\cE$ to a general hyperplane section $Y$ of $X$ is not initialized.
   Indeed, let $\cE$ be the rank 2 vector bundle on $\PP^3$  given by the short exact sequence:
   $$
   e: \quad \quad  0 \longrightarrow \cO _{\PP^3}\longrightarrow \cE  \longrightarrow I_Z(2) \longrightarrow 0$$
   where $0\ne e\in Ext^1(I_Z(2),\cO_{\PP^3})$ and $Z=L_1\cup L_2$ is the union of two disjoint lines $L_1$ and $L_2$. It holds: $H^0(\PP^3,\cE(-1))=0$ and $H^0(\PP^3,\cE)\ne 0$. Therefore, $\cE$ is initialized.
Let us consider its restriction $\cE_{H}$ to a general plane $H\subset \PP^3$. We have:
 $$
   e_H: \quad \quad  0 \longrightarrow \cO _{H}\longrightarrow \cE_H  \longrightarrow I_{Z\cap H}(2) \longrightarrow 0.$$
  Since $H^0(H,I_{Z\cap H}(1))=\kk$ (because there is a line $L$ through the two different points of $Z\cap H$), we get $H^0(H,\cE_{|H}(-1))=\kk \ne 0$ and, hence, $\cE_{|H}$ is not initialized.
So, the previous result highlights a very important property of Ulrich bundles.
\end{remark}

\begin{lemma}\label{keylemma}
     Let $X$ be a smooth irreducible projective variety  with a fixed ample
line bundle $\cO _X (1)$ and let $\cE$ be an  initialized Ulrich bundle of rank $r$ on $X$.
Let $Y\subset X$ a general hyperplane section. It holds $S_{\cE_{|Y}}\cong (S_{\cE})_{|Y}.$
\end{lemma}

\begin{proof} The proof relies on the following observation: Since  $\cE$ is a rank $r$ Ulrich bundle on $X$ we have $H^0(X,\cE(-1))=H^1(X,\cE(-2))=0$ and from the exact cohomology sequence associated to
 $$
    0 \longrightarrow \cE (-1)\longrightarrow \cE  \longrightarrow \cE _{|Y } \longrightarrow 0 $$
    we  get $H^0(X,\cE)\cong H^0(Y,\cE _{|Y})=\kk ^{r\cdot \deg(X)}$. We consider now the short exact sequence
    $$
    0 \longrightarrow S_{\cE}\longrightarrow H^0(X,\cE)\otimes \cO_X  \longrightarrow \cE  \longrightarrow 0 $$
    whose restriction to $Y$ gives rise to the following commutative diagram:

    \[
\xymatrix{0\ar[r] & (S_{\cE})_{|Y}  \ar[r] & H^0(X,\cE)\otimes \cO_Y \ar[r] \ar[d]^{\psi } & \cE _{|Y } \ar[d]^{\varphi} \ar[r] & 0 \\
0\ar[r] & S_{\cE_{|Y}}  \ar[r] & H^0(Y,\cE_{|Y})\otimes \cO_Y \ar[r] & \cE _{|Y } \ar[r] & 0
}
\]
Applying the snake lemma and taking into account that $\psi $ and $\varphi $ are isomorphisms, we deduce that $S_{\cE_{|Y}}\cong (S_{\cE})_{|Y}$ which proves what we want.
\end{proof}

\begin{remark} \rm
   Again the above lemma highlights  an important feature of Ulrich bundles since, in general, it is far from being true that the restriction $(S_{\cE})_{|H}$ to a general hyperplane section  $H$ of a syzygy bundle $S_{\cE}$ associated to a globally generated  vector bundle $\cE$ is the syzygy bundle $S_{\cE_{|H}}$ of the restriction to $H$ of  the vector bundle $\cE$. In fact, we consider the cotangent bundle $\cE=\Omega _{\PP^n}^1(1)$, i.e. the syzygy bundle associated to the line bundle $\cO_{\PP^n}(1)$. Its restriction to a general hyperplane $H\cong \PP^{n-1}\subset \PP^n$ is  $(\Omega^1_{\PP^{n}}(1))_{|\PP^{n-1}}\cong \Omega^1_{\PP^{n-1}}(1)\oplus \cO_{\PP^{n-1}}$. Therefore, $(S_{\cO_{\PP^n}(1)}) _{|\PP^{n-1}}\cong S_{\cO _{\PP^{n-1}}(1)}\oplus \cO_{\PP^{n-1}}$.
\end{remark}

\begin{theorem}\label{main}
 Let $X$ be a smooth  Fano variety of dimension $n$ and index $i_X\ge n-2$ or a smooth K3 surface. Let $\cE$ be a rank $r$  Ulrich bundle on $X$. Then, the syzygy bundle $S_{\cE}$ is semistable.
\end{theorem}

    \begin{proof}
Let us first compute the sectional genus $g$ of $X$ with respect to $H\in |\cO_X(1)|$. Repeatedly applying the adjunction formula, we can write the sectional genus as
$$
g=1+\frac{n-1}{2}H^n-\frac{c_1(X)\cdot H^{n-1}}{2}
$$
where $g$ is the genus of the smooth irreducible curve $H^{n-1}$, $H^n$ its degree and $c_1(X)$ the first Chern class of $X$, i.e.,  the first Chern class of the tangent bundle $T_X$ of $X$ and hence $c_1(X)=-K_X$.

\noindent
{\em Claim: } The restriction $(S_{\cE})_{|H^{n-1}}$ of $S_{\cE}$ to an irreducible  smooth curve $H^{n-1}\subset X$ is semistable.

\noindent
{\em Proof of the Claim:} We first observe  that, by repeatedly applying Lemma \ref{keylemma}, we have:
$(S_{\cE})_{|H^{n-1}}\cong S_{\cE_{|H^{n-1}}}$. On the other side, the slope of any initialized Ulrich bundle on a smooth curve $H^{n-1}$ of degree $H^{n}$ and genus $g$ satisfies the equality \cite[Lemma 3.2.4]{CMP}:
$$
\mu (\cE_{|H^{n-1}})=H^n+g-1.$$
Therefore, we have
$$
\begin{array}{rcl}
\mu (\cE_{|H^{n-1}}) & = & H^n+g-1 \\
& = & H^n+\frac{n-1}{2}H^n-\frac{c_1(X)\cdot H^{n-1}}{2} \\
& \ge & 2g
\end{array}$$ where in the last inequality we use that $n=2$ and $c_1(X)=0$ or $n>2$ and $c_1(X)=-K_X=i_XH\ge (n-2)H.$ Since $\cE_{|H^{n-1}}$ is a semistable, globally generated vector bundle (Proposition \ref{properties} (i) and (iii) ) on a smooth curve $H^{n-1}$ of genus $g$  and $\mu (\cE_{|H^{n-1}})\ge 2g$;  we can apply \cite[Theorem 1.2]{B94} and conclude that the syzygy bundle $S_{\cE_{|H^{n-1}}}$ is semistable which proves the Claim.

Assume that $S_{\cE}$ is not semistable and let $\cF\subset S_{\cE}$ be a destabilizing subbundle of rank $0<\rk(\cF)<\rk(S_{\cE})$ such that $$\mu(\cF)=\frac{c_1(\cF)\cdot H^{n-1}}{\rk(\cF)}>\mu(S_{\cE})=\frac{c_1(S_{\cE})\cdot H^{n-1}}{\rk(S_{\cE})}.$$
Restricting to a general smooth curve $H^{n-1}$ we obtain $\mu (\cF_{|H^{n-1}})>\mu ((S_{\cE})_{|H^{n-1}})$ contradicting the semistability of $(S_{\cE})_{|H^{n-1}}$.
    \end{proof}

As an immediate application of the above theorem we get:

\begin{corollary}\label{Fano3fold}
    Let $X$ be a Fano 3-fold and  let $\cE$ be a rank $r$  Ulrich bundle on $X$. Then the syzygy bundle $S_{\cE}:=Ker(eval: H^0(X,\cE)\otimes \cO_X\longrightarrow \cE)$ is semistable.
\end{corollary}

\begin{remark} \rm Let $X$ be a smooth irreducible projective variety  with a fixed ample
line bundle $\cO _X (1)$. It is not always true that the syzygy bundle $S_{\cE}$ associated to an Ulrich bundle $\cE$ on $X$ is semistable. In fact, we know that the slope of any Ulrich bundle on a smooth
projective curve $C$ of degree $d$ and genus $g$ is $\mu (\cE)=d+g-1$ \cite[Lemma 3.2.4]{CMP}. On the other hand, by \cite[Theorem 1.2]{B94}, the syzygy bundle $S_{\cE}$ associated to a semistable globally generated vector bundle $\cE$ on $C$ is semistable if $\mu (\cE)\ge 2g$. Therefore, the syzygy bundle $S_{\cU}$ associated to a rank 2 Ulrich bundle $\cU$ on a smooth plane quintic curve $C$ is not semistable. However,  for any integer $m\ge 1$, the   syzygy bundle $S_{\cU(m)}$ associated to  $\cU(m)$ is stable.
\end{remark}

Based on our results and the above remark we  state the following guess:

\begin{conjecture}
    Let $X$ be a smooth irreducible projective variety  with a fixed ample
line bundle $\cO _X (1)$ and let $\cE$ be an Ulrich  bundle on $X$. There exists an integer $m¨_0\gg 0$ such that for all $m\ge m_0$ the syzygy bundle $S_{\cE(m)}$ associated to $\cE(m)$ is semistable.
\end{conjecture}

\vskip 4mm

 We have seen that the syzygy bundle $S_{\cE}$ of a rank $r$ Ulrich bundle on  a smooth Fano variety $X$ of dimension $n$ and index $i_X\ge n-2$  is semistable. Therefore it represents a point inside a suitable moduli space $\cM(\rk(S_{\cE});c_{i}(S_{\cE}))$ of semistable  vector bundles of fixed rank and Chern classes. We will end this short note  showing that they are parametrized by an open dense subset of a generically smooth irreducible component of $\cM(\rk(S_{\cE});c_{i}(S_{\cE}))$.

\vskip 2mm

To this end, we fix a smooth Fano variety $X$ of dimension $n\ge 3$ and index $i_X\ge n-2$ and a rank $r$ stable Ulrich $\cE$ bundle on $X$ with Chern classes $c_i:=c_i(\cE)$. We denote by
 $\Spl(r;c_i)$ the moduli space of simple
sheaves  $\cE$ (i.e. $\End (\cE)=\kk$)  on $X$ of rank $r$ and Chern classes $c_i$ and by $\cM(r; c_i)$ the open subset parameterizing semistable sheaves.
 We  consider $$U:=U(r;c_i)\subset   \Spl(r; c_i)$$ the open locus parameterizing  globally generated rank $r$ simple vector bundles $\cF$ on $X$ such that $H^1(X,\cF)=H^1(X,\cF^\lor)=H^2(X,\cF^\lor)=0$.
 Recall that  for any Ulrich bundle $\cE$ on $X$ we have $H^1(X,\cE)=H^1(X,\cE^\lor)=H^2(X, \cE^\lor)=0$. Therefore, the open subset $U(r;c_i)$ is non-empty and we will call $\cU l (r;c_i)$ the non-empty open dense (in suitable components) parameterizing Ulrich bundles; it will be also dense inside  suitable components of  $\Spl(r;c_i)$.

For any  $\cE\in \cU l(r;c_i)$, the associated  syzygy bundle $S_{\cE}$ is simple and semistable (Theorem \ref{main}) and, using stack theory, we can define  a morphism (see \cite[Definition 3.7]{FM} and \cite[Definition 3.12]{FM} for details)
 $$\alpha:\cU l(r;c_i) \longrightarrow \Spl(\rk(S_{\cE});c_i(S_{\cE}))$$ which extends the set-theoretic map $\cE\mapsto S_{\cE}$. It is known that:
\begin{itemize}
    \item[(i)]
 The morphism $\alpha$ is injective and a locally closed embedding \cite[Theorem 1.1 and Proposition 4.1]{FM1} and \cite[Proposition 3.13]{FM};
 \item[(ii)] $\alpha $ is an open embedding. Indeed, it follows from  \cite[Theorem  1.2]{FM1} taking into account that $dim X\ge 3$ and $H^2(X,\cO_X)=0$;
\item[(iii)] We have $$T_{[(\cE]}\cU l(r;c_i)=H^0(X,\cE^{\vee}\otimes \cE),$$
$$T_{[S_{\cE}]}\Spl (\rk(S_{\cE});c_i(S_{\cE}))=\Ext^1(S_{\cE},S_{\cE})=H^1(X,S_{\cE}^{\vee}\otimes S_{\cE})$$
and the differential map  $$d \alpha :T_{[\cE]}\cU l(r;c_i)\longrightarrow T_{[S_{\cE}]}\Spl (\rk(S_{\cE});c_i(S_{\cE}))$$ is injective (see \cite[Proposition 4.2]{FM1} and \cite[Proposition 3.13]{FM}).
 \end{itemize}

We will denote by $Syz_{\cU l} (\rk(S_{\cE});c_i(S_{\cE}))$ the image of the morphism $\alpha $ and we call it the {\em moduli space of the syzygy bundles}.
Putting altogether we have:

\begin{theorem} \label{Theorem:moduli} Let $X$ be a smooth Fano  variety of dimension $n\ge 3$ and index $i_X\ge n-2$. Let $\cU l(r;c_i)\subset \cM(r;c_i)\subset \Spl(r;c_i)$ be the moduli space of rank $r$ stable Ulrich bundles on $X$
with fixed Chern classes $c_i$. Assume $\cU l(r;c_i)$ non-empty and smooth. Then, $Syz_{\cU l} (\rk(S_{\cE});c_i(S_{\cE})) \subset \Spl(\rk(S_{\cE});c_i(S_{\cE}))$ is a smooth
open subspace and $dim _{\kk}\cU l(r;c_i)=dim _{\kk} Syz_{\cU l} (\rk(S_{\cE});c_i(S_{\cE})).$
\end{theorem}



\begin{thebibliography}{99}

\bibitem{BP} S. Basu and    S. Pal, {\em Stability of Syzygy bundles corresponding to stable vector bundles
on algebraic surfaces}, Bull. Sci. Math. {\bf 189} (2023), 103358.

\bibitem{Beau} A. Beauville, {\em An introduction to Ulrich bundles}, Eur. J. Math. {\bf 4} (2018),
26–36.



\bibitem{B94} D.C. Butler, {\em Normal generation of vector bundles over a curve}, J. Differential Geom.,
{\bf 39} (1994), 1-34.

\bibitem{CL} F. Caucci and M. Lahoz, {\em Stability of syzygy bundles on abelian varieties}, Bull. Lond. Math. Soc. {\bf 53:4}  (2021) 1030-1036.



\bibitem{CMP} L. Costa, R. M. Mir\'o-Roig, and J. F. Pons-Llopis, {\em Ulrich bundles: From
commutative algebra to algebraic geometry}, De gruyter Studies in Math-
ematics {\bf 77}, De Gruyter, Berlin (2021).

\bibitem{EL} L. Ein and R. Lazarsfeld, {\em  Stability and restrictions of Picard bundles, with an application to the normal bundles
of elliptic curves}. Complex projective geometry (Trieste, 1989/Bergen, 1989), 149–156, London Math. Soc. Lecture Note
Ser., {\bf 179} (1992), Cambridge Univ. Press, Cambridge.

\bibitem{ELM} L. Ein, R. Lazarsfeld and Y. Mustopa, {\em Stability of syzygy bundles on an algebraic surface}. Math. Res. Lett. {\bf 20:1} (2013), 73--80.



\bibitem{ES} D. Eisenbud and F.O. Schreyer, {\em Resultants and Chow forms via exterior syzygies}, J. Amer. Math. Soc. {\bf 16} (2003), 537–579.

\bibitem{FM} B. Fantechi and R. M. Mir\'o-Roig, {\em Lagrangian subspaces of the moduli space of simple sheaves on K3 surfaces},  Mediterr. J. Math. 22, 21 (2025). https://doi.org/10.1007/s00009-024-02791-1.

\bibitem{FM1} B. Fantechi and R. M. Mir\'o-Roig, {\em Moduli space of generalized syzygy bundles},Pure and Applied Mathematics Quarterly 20 (2024) 2113-2145.  https://dx.doi.org/10.4310/pamq.

\bibitem{Fle} H. Flenner, {\em  Restrictions of semistable bundles on projective varieties}, Comment. Math. Helv. {\bf 59} (1984),  635-650.


\bibitem{G} M. Green, {\em Koszul cohomology and the geometry of projective varieties}, J. Differential Geom. {\bf 19} (1984),  125-171.



\bibitem{MRay} S. Misra and N. Ray, {\em On the stability of syzygy bundles}. Preprint arXiv 2405.17006v1.

\bibitem{MR} J. Mukherjee and  D. Raychaudhury, {\em A note on stability of syzygy bundles on Enriques and bielliptic surcaes}, arXiv 2111.08231.

\bibitem{R} N. Rekuski, {\em Stability of kernel sheaves associated to rank one torsion free sheaves}. Math. Z. {\bf 307} (2024).

\bibitem{S}  I. R. Shafarevich. {\em Algebraic geometry V}, volume 47 of Encyclopaedia of Mathematical Sciences. Springer-Verlag, Berlin, 1999. Fano varieties, A translation of Algebraic geometry 5 (Russian), Ross. Akad. Nauk, Vseross. Inst. Nauchn. i Tekhn. Inform., Moscow, Translation
edited by A. N. Parshin and I. R. Shafarevich.


\bibitem{Ul} B. Ulrich, {\em Gorenstein Rings and Modules with High Numbers of Generators}, Math. Z. {\bf 188}
(1984), 23-32



\end{thebibliography}
\end{document}